\documentclass[reqno]{amsart}
\newcommand \datum {February 23, 2024}

\usepackage{amssymb,latexsym}
\usepackage{amsmath}
\usepackage{hyperref}
\usepackage{url}
\usepackage{graphicx}
\usepackage[dvipsnames]{xcolor}
\usepackage{enumerate}
\numberwithin{equation}{section}
\theoremstyle{plain}
 \newtheorem{theorem}{Theorem}[section]
  
 \newtheorem{lemma}[theorem]{Lemma}

 \newtheorem{corollary}[theorem]{Corollary}
\theoremstyle{definition}
 \newtheorem{definition}[theorem]{Definition}
 \newtheorem{nextpart}{Part}
 \newtheorem{example}[theorem]{Example}
 \newtheorem{remark}[theorem]{Remark}
 \newtheorem{para}{$\bullet$}
 \newtheorem{tara}{$\scriptscriptstyle\blacksquare$}
\theoremstyle{remark}

\newcommand \Pow [1]{\textup{Pow}(#1)}
\newcommand \restrict[2] {#1\rceil_{\kern-1pt #2}}

\newcommand \strLASp {\textup{LASp}}
\newcommand \LASp[1] {\strLASp(#1)}
\newcommand\Quo[1]{\textup{Quo}(#1)}
\newcommand\Quleq{\textup{Quo}^{\textup{e}}}
\newcommand\Pquleq {\Quleq(\PP)}
\newcommand\RQuleq[1]{\textup{Quo}_{#1}^{\textup e}(\PP)}
\newcommand\Equ[1]{\textup{Equ}(#1)}
\newcommand \filter[1]{\textup{fil}(#1)}
\newcommand \ideal[1]{\textup{idl}(#1)}
\newcommand \pfilter[2]{\textup{fil}_{#1}(#2)}
\newcommand \pideal[2]{\textup{idl}_{#1}(#2)}

\newcommand \Max [1] {\textup{Max}(#1)}

\newcommand \ncs {{\textup{csize}(\PP)}}
\newcommand \ncmp {{\textup{cmp}(\PP)}}
\newcommand \ncedge {{\textup{ce}(\PP)}}

\newcommand \spp  {{\textup{spp}}}
\newcommand \ffunc {f_4}
\newcommand \thn {\textup{thn}(\PP,\bSelc_1,\bSelc_2)}

\newcommand \cop{\textup{cp}(\PP)}

\newcommand \Pdnul {P^{(0)}_{\bSelc_1,\bSelc_2}}
\newcommand \PPdnul {\mathbb P^{(0)}_{\bSelc_1,\bSelc_2}}

\newcommand \Edge [1] {\textup{Edge}(#1)}
\newcommand \iEdge [1] {\textup{Edge}^{-1}(#1)}
\newcommand \Comp [1] {\textup{Comp}(#1)}
\newcommand \Strip{\textup{Stp}(\PP)}
\newcommand\Thread{\textup{Thd}(\PP)}

\newcommand \nab [1] {\nabla_{{\kern-1.5pt #1}}}
\newcommand \enab [1] {\nabla^{+}_{{\kern-1.5pt #1}}}
\newcommand \spech {h}
\newcommand \bup{\beta^{\textup{up}}}
\newcommand \gdn{\gamma_{\textup{dn}}}
\newcommand \Csum {\sum^{\textup{card}}}
\newcommand \cplus {\mathrel{+^{\textup{card}}}}

\newcommand \nothing[1]{}
\newcommand \Min [1] {\textup{Min}(#1)}

\newcommand \str[1] {#1^-}
\newcommand \mbigvee {\bigvee{}\kern-2pt^\mu}
\newcommand \rtr {\textup{rtr}}
\newcommand \quo {\textup{quo}}
\newcommand \qum {{\textup{qu}_{\mu}}}

\newcommand \Selc {M}
\newcommand \bSelc {\mathbb M}
\newcommand \SSelc {M^{\textup{sl}}}
\newcommand \ul[1]{\underline{#1}}

\renewcommand \phi{\varphi}
\newcommand \Nnul {\mathbb N_0}
\newcommand \Nplu {\mathbb N^+}
\newcommand \then {\Rightarrow}

\newcommand{\tbf}{\textbf}
\newcommand{\set}[1]{\{#1\}}
\newcommand \PP {{\mathbb P}}
\newcommand \bbA {{\mathbb A}}
\newcommand \XX {{\mathbb X}}
\newcommand \TT {{\mathbb T}}

\newcommand \red[1]{{\textcolor{red}{#1}\color{black}}}

\newcommand \teal[1]{{{\textcolor{teal}#1}}}

\newcommand \tul[1] {\begin{tara}#1 \end{para}}
\newcommand \atul[2] {\begin{t0ara}[{\teal{{about your hint ``#1{}''}}}] {#2} \end{tara}}

\begin{document}

\title[Quasiorder lattices]
{Large filters of quasiorder lattices can be generated by few elements}

\author[G.\ Cz\'edli]{G\'abor Cz\'edli}
\email{czedli@math.u-szeged.hu}
\urladdr{http://www.math.u-szeged.hu/~czedli/}
\address{University of Szeged, Bolyai Institute. 
Szeged, Aradi v\'ertan\'uk tere 1, HUNGARY 6720}

\begin{abstract} 
For a poset $(P;\leq)$, the quasiorders (AKA preorders) extending the poset order ``$\leq$''  form a complete lattice $F$, which is a filter in the lattice of all quasiorders of the set $P$. We prove that if the poset order ``$\leq$'' is small, then $F$ can be generated by few elements. 
\end{abstract}

\thanks{This research was supported by the National Research, Development and Innovation Fund of Hungary, under funding scheme K 138892.}

\subjclass {06B99\hfill{\red{\tbf{\datum}}}}

\dedicatory{Dedicated  to Professor Vilmos Totik on his retirement (2023) and
forthcoming seventieth birthday (2024); see also Remark \ref{rem:totik} in Section \ref{sect:dedic}.}


\maketitle

\section{Introduction}\label{S:Introduction}
Before stating the main result,  Theorem \ref{thmmain}, we survey the previously known results on generation of quasiorder lattices. 
First of all, we recall some concepts and notations. 
Reflexive and transitive relations of a set $X$ are called  \emph{quasiorders} or (in a significant portion of the literature) \emph{preorders} (of $X$). For a  set $X$, we will denote by $\Quo X$ the complete lattice of quasiorders of $X$; the lattice order ``$\leq$'' in $\Quo X$ is the set inclusion relation ``$\subseteq$''. 
A subset $Y$ of a complete lattice $L$ is a \emph{complete generating set} if no proper complete sublattice of $L$ includes $Y$. Furthermore, we say that $L$ is \emph{$n$-generated as a complete lattice} if $L$ has an at most $n$-element
complete generating set.
For a finite lattice, this concept is equivalent to the existence of an at most $n$-element generating set in the usual sense. E.g., if $|L|\leq n$, then $L$ is  $n$-generated.
As a convention for this paper, when we write that  ``generated'', we always mean ``generated as a complete lattice'' even when this is not emphasized.

To recall the first result on small generating sets of $\Quo X$ from   Chaj\-da and Cz\'edli \cite{chajdaczg}, let  $\beth_0:=\aleph_0$ and, for any integer $n\in\Nnul=\set{0,1,2,\dots}$, let $\,\beth_{n+1}:=2^{\beth_n}$. Let $\beth_\omega:=\sup\set{\beth_n:n\in\Nnul}$. With these notations, if $|X|\leq \beth_\omega$, then $\Quo X$ is 6-generated (but note that the six-element generating set given in \cite{chajdaczg} has an additional property).

A cardinal number $\eta>\aleph_0$ is \emph{strongly inaccessible} if  for any cardinal $\lambda$, \ $\lambda<\eta\then 2^\lambda<\eta$ and, in addition, for any set $I$ of cardinals smaller than $\eta$, \ $|I|<\eta \then \sup I<\eta$.
We call a cardinal number $\lambda$ \emph{accessible} if there exists no strongly inaccessible cardinal number $\eta$ such that $\eta\leq \lambda$.
For example, the finite cardinals and $\beth_\lambda$ for $\lambda\leq \omega$ are accessible. 
We know from Kuratowski's result, see \cite{kuratowski} and see also Levy \cite{levy} for a secondary source that if ZFC has a model, then it also has a model in which all cardinal numbers are accessible.

As the next step after Chajda and Cz\'edli \cite{chajdaczg}, Tak\'ach \cite{takach} extended the result of  \cite{chajdaczg} to every set $X$ of accessible cardinality.
About two decades later, Dolgos \cite{dolgos} proved that $\Quo X$ is 5-generated for any set $X$ with size (=cardinality) $|X|\leq \aleph_0$, and Kulin \cite{kulin} extended this result by proving that $\Quo X$ is 5-generated for any set $X$ with an accessible cardinality. Not much later, with the exception of $|X|\in\set{4, 6, 8, 10}$, Cz\'edli \cite{czg2017fourgen} proved that if $X$ is a set with $|X|\leq \aleph_0$, then $\Quo X$ is 4-generated. Furthermore, 
 \cite{czg2017fourgen} also proved that for any set $X$ with at least three elements, $\Quo X$ is not 3-generated. 
Soon thereafter, Cz\'edli and Kulin \cite[Theorem 3.5]{czgkulin} reduced the number of exceptions and permitted all infinite sets $X$ with accessible cardinalities. We summarize these results  as follows.

\begin{lemma}[Cz\'edli \cite{czg2017fourgen}, Cz\'edli and Kulin \cite{czgkulin}]
\label{lemma:czkLn}
If $X$ is an at least two-element  set of an accessible cardinality and this cardinality is different from $4$, then the complete lattice $\Quo X$ is $4$-generated as a complete lattice. For $|X|\geq 3$,  $\Quo X$, as a complete lattice, is not $3$-generated. For $|X|=4$, $\Quo X$ is $5$-generated.
\end{lemma}

For $|X|=4$, we do not know  whether $\Quo X$ is $4$-generated or not.  Our understanding becomes less comprehensive when we stipulate that two out of the four generators are comparable. What we know from Cz\'edli and Kulin \cite[Theorem 3.5]{czgkulin} and Ahmed and Cz\'edli \cite[Remark 3.4]{delbrinczg} is that whenever $X$ is a set such that $|X|$ is an accessible cardinal but $|X|\notin\set{0,1,2,4, 5, 7, 8, 9, 10, 12}$, then $\Quo X$ has a $4$-element generating set that is not an antichain.

For equivalence lattices, there are  results similar to Lemma \ref{lemma:czkLn} and the paragraph following it; see Cz\'edli \cite{CzGEq4}, \cite{CzGEateq}, \cite{czgeek}, \cite{czgdaugthent}, 
Cz\'edli and Oluoch \cite{czgoluoch},
Dolgos \cite{dolgos}, 
Strietz \cite{strietz75}--\cite{strietz77}, and 
Z\'adori \cite{zadori}. Most of the results and papers mentioned so far are based on Z\'adori's excellent method given in \cite{zadori}, in which one of the theorems asserts that the lattice $\Equ X$ of equivalence relations of $X$ is $4$-generated for every finite set $X$. 
Note that although equivalence lattices are less complicated structures than quasiorder lattices,  it was quasiorder lattices that showed the way how to extend results  from finite sets to some infinite sets; see  Chajda and Cz\'edli  \cite{chajdaczg}.

Recently, Cz\'edli \cite{czgdaugthent} and, mainly, \cite{czgBoolegen} have suggested (but not elaborated)  cryptographic protocols based on large lattices generated by few elements. This idea is also among our motivations, since a straightforward induction shows that for $n\in\Nplu =\set{1,2,3,\dots}$, the quasiorder lattice $\Quo{\set{1,\dots,n}}$ has at least $4^{n-1}$ elements, so it is large.

\section{Concepts, notations, and the main result}

First, we recall or introduce the concepts and notations that are necessary to formulate  the main result. As is usual in lattice theory, a lattice $(L;\vee,\wedge)$ or $(L;\leq)$ is denoted simply by its underlying set, $L$. 
However, as a poset and its underlying set will often play different roles, we make a notational difference between them as follows. A poset $(X;\leq)$ or $(X;\mu)$ will be denoted by the corresponding  blackboard bold letter $\XX$, and the underlying set of a poset $\mathbb Y$ is the corresponding italic letter $Y$.
Sometimes, we write $u\in \XX$ or $Y\subseteq \XX$ instead of  $u\in X$ or $Y\subseteq X$, respectively.

For an element $u$ of a lattice $L$, we denote the \emph{principal filter} $\set{x\in L: u\leq x}$ by $\filter u$ or, if confusion threatens, by $\pfilter L u$. Similarly, $\ideal u = \pideal L u$ stands for the  \emph{the principal ideal}  $\set{x\in L: x\leq u}$.

Posets with more than one element are said to be \emph{nontrivial} while the singleton poset is \emph{trivial}. The \emph{length} %
of a finite $n$-element chain is $n-1$. For $n\in\Nnul$,  a \emph{poset is of length} $n$ if it has a chain of length $n$ but it has no chain of length $n+1$.   If $\PP =(P;\leq)$ is a poset and it is of length $n$ for some $n\in\Nnul$, then we say that $\PP $ is \emph{of finite length}. 

We always assume that our poset $\PP $ is of finite length; in this case, the covering relation $\prec$ of $\PP $ determines the (partial) order $\leq$ of $\PP $. While the Hasse diagram of $\PP $ is a directed graph (every edge is directed upwards but this is not indicated), this diagram can also be considered an undirected graph,  the \emph{graph of the poset} $\PP $. The edges of this graph are the two-element subsets $\set{x,y}$ of $\PP $ such that $x\prec y$ or $y\prec x$. We say that the poset $\PP $ is a \emph{forest} if its graph contains no circle of positive length. In this case, the \emph{$($connectivity$)$ components} of $\PP $ (that is, the blocks of the least equivalence of $\PP $ that collapses every edge of $\PP$) are called \emph{trees}. 
For a poset $\XX$, we denote the \emph{set of maximal elements} as $\Max \XX$ and the  \emph{set of minimal elements} as $\Min \XX$.

Next, we define the functions and the parameters that occur in the main result, Theorem \ref{thmmain}; most of them will be denoted by self-explanatory 
mnemonics\footnote{\label{foot:cSrmNk}We have selected mnemonics that can be conveniently located in the  PDF file of the paper using the  search feature of Acrobat Reader or some browsers. For example, the search for ``ce(''  or ``f4''finds the first occurrence of $\ncedge$ or $f_4$, respectively}. 
The parameters are illustrated by Figures \ref{fig1} and \ref{fig2} together with the corresponding Examples \ref{exmpl:BlTrsst} and \ref{exmpl:jtKgmbRnk}, respectively; note that some ingredients of these multi-purpose figures will be defined later.
The main result has some  easy-to-read consequences with less functions and parameters; see Corollary \ref{cor:zsbPl} and Examples \ref{exmpl:BlTrsst}, \ref{exmpl:jtKgmbRnk}, and  \ref{ex:mPkrDsz}. 

%

\begin{figure}[ht] 
\centerline{ \includegraphics[width=\textwidth]{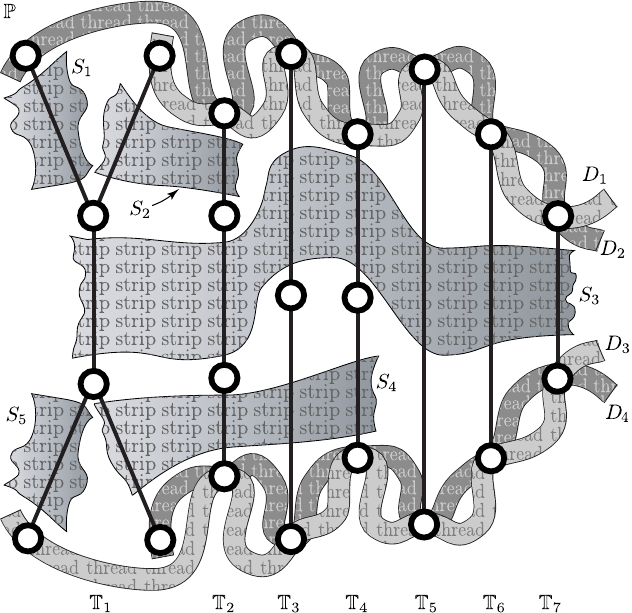}} \caption{A forest $\PP $; see Example \ref{exmpl:BlTrsst}. 
}\label{fig1}
\end{figure}

\begin{definition}\label{def:fntnS} 
For a positive integer $n$, let $\LASp n$ 
denote\footnote{The acronym LASp comes from the \ul Left \ul Adjoint of \ul{Sp}erner's function, which is a functor from the categorified poset $(\set{0,1,2,\dots};\leq)$ to itself; see  Cz\'edli \cite{czgBoolegen} for details.} 
the smallest  $k\in\Nplu:=\set{1,2,3,\dots}$ such that 
\begin{equation}
n\leq { k \choose {\lfloor k/2 \rfloor }}
\label{eq:prSnvlTsrvkZkmln}
\end{equation}
where $\lfloor k/2 \rfloor$ denotes the (lower) integer part of $k/2$, and let $\LASp 0:=0$. 
For some values of $n$, \ $\LASp n$ is given in the following tables.

\begin{tabular}{|l|l|}
\hline
$n$ &$\LASp n$\\\hline
$0,\dots,3$ & $n$\\\hline
$4,\dots,6$ & 4\\\hline
$7,\dots,10$ & 5\\\hline
$11,\dots,20$ & 6 \\\hline
$21,\dots,35$ & 7\\\hline
$36,\dots,70$ & 8\\\hline
$71,\dots,126$ & 9\\\hline
$127,\dots,252$ & 10\\\hline
$253,\dots,462$ & 11\\\hline
$92\,379,\dots,184\,756$ & 20\\\hline
\end{tabular}
\kern 2cm
\begin{tabular}{|l|l|}
\hline$n$\kern 0.8cm&$\LASp n$\\
\hline
$10^{10}$& 37\\\hline
$10^{20}$& 70\\\hline
$10^{30}$ & 104\\\hline
$10^{40}$ & 137 \\\hline
$10^{50}$ & 171\\\hline
$10^{60}$ & 204\\\hline
$10^{70}$ & 237\\\hline
$10^{80}$ & 271\\\hline
$10^{90}$ & 304\\\hline
$10^{100}$ & 337\\\hline
\end{tabular}  
\end{definition}

%

\begin{definition}\label{def:fngfnCt} 
We define a function $\ffunc\colon\set{\text{accessible cardinals}}\setminus\set 0\to\Nnul$ by the rules $\ffunc(1):=0$, $\ffunc(2):=2$, $\ffunc(3):=4$,
$\ffunc(4):=5$ and, for any accessible cardinal $x\geq 5$, $\ffunc(x):=4$. (The subscript is only a reminder to the fact that $\ffunc(x)=4$ for most cardinal numbers $x$.)
\end{definition}


\begin{definition} \label{def:bzSncsGnd} 
Let $\PP =(P;\leq)=(P;\mu)$ 
be a poset  of finite length.

(A) Let $\Comp \PP $ stand for the set of  (connectivity) \emph{components} of $\PP $. 
If the equality $|\Comp\PP|=1$ holds, then $\PP$ is a \emph{connected poset}.
We say that $\PP $ is of \emph{finite component size} if there is a $k\in\Nplu:=\Nnul\setminus\set{0}$
such that for every $\XX\in \Comp \PP $, $\XX$ has at most $k$ elements (that is, $|X|\leq k$). 
In this case, the least such $k$ is called the \emph{component size} of $\PP $; we denote it by $\ncs$.

(B) Let $\ncmp=|\Comp \PP |$; it is the cardinal number of the components of $\PP $. 

(C) Let $\Edge \PP :=\set{(x,y)\in P^2: x\prec y}$, where $\prec$ is the covering relation of $\PP$,  stand for the \emph{set of edges} of $\PP $.

(D) Assuming that $\PP $ is of a finite component size, $\ncs^2$ is a finite upper bound on $\set{|\Edge {\XX}|: \XX\in \Comp \PP }$ and so we can define the 
\emph{component edge  number} of $\PP $ as $\ncedge:=\max\set{|\Edge {\XX}|: \XX\in \Comp \PP }$.  (Note that even though $\ncedge\in\Nnul$,  $\PP $ can have  infinitely many edges.)

(E) For a poset $\bbA=(A;\mu_A)=(A;\leq)$ and $x,y\in A$,  let
\[\qum(y,x) \in\Quleq(\bbA)
\] 
stand for the least member of $\Quleq(\bbA)$ containing $(y,x)$. For a finite connected poset $\XX=(X;\mu_X)=(X;\leq)$, we define the \emph{special parameter} $\spp(\XX)$ as follows. If $X$ is a singleton, then $\spp(\XX):=0$. If $\XX$ is a chain with more than one element, then $\spp(\XX):=1$. 
If $\XX$ is not a chain, then  $\spp(\XX)$ is the smallest number $k\in\Nplu$ such that there is a $(k-1)$-element subset $
Y(\XX)$ of  $\Quleq(\XX)$ with the property that 
\begin{equation}Y(\XX)\cup\set{\qum(y,x): x\prec y\text{ in }\XX}\text{ generates }\Quleq(\XX).
\label{eq:ZlsJdlFRtXgmnRt}
\end{equation}
We do not claim that $Y(\XX)$ above is unique but, for each $\XX$, we always work with a fixed one satisfying \eqref{eq:ZlsJdlFRtXgmnRt}.

(F) Assume that $\ncmp=|\Comp\PP|$ is at least  $2$, and let   $\bSelc_1$ and $\bSelc_2$ be two distinct components of $\PP$. Let $\PPdnul$ 
denote  the poset that we obtain from $\PP$ by omitting $\bSelc_1$ and $\bSelc_2$. That is, $\PPdnul$ is the subposet of $\PP$ determined by $\Pdnul=P\setminus(\Selc_1\cup\Selc_2)$. If $\ncmp=2$ then, exceptionally, $\PPdnul$ is the ``empty poset''. We define the \emph{thread number} $\thn$ as follows.
\begin{align}
\text{If }\PP&\text{ is an antichain, then  }\thn:=1.
\label{eq:mrNgsVzTsTrpmn}
\\
\text{If }\PP&\text{ is not an antichain and }\ncmp =2\text{, then}\cr
&\thn:=\prod_{i=1}^2|\Max{\bSelc_i}| + \prod_{i=1}^2|\Min{\bSelc_i}|.
\label{eq:kNgbRkssRjb}
\\
\text{If }\PP  &\text{ is not an antichain and }\ncmp \geq 3\text{, then }
\cr
&\thn:=\cr
&\max\set{|\Max {\XX}|: {\XX}\in \Comp \PPdnul }\cdot\prod_{i=1}^2|\Max{\bSelc_i}| +\phantom{,}
\cr
&\max\set{|\Min {\XX}|: {\XX}\in \Comp \PPdnul} \cdot\prod_{i=1}^2|\Min{\bSelc_i}|.
\label{eq:mMkflgcNbFpB}
\end{align}
Based on \eqref{eq:mrNgsVzTsTrpmn}, \eqref{eq:kNgbRkssRjb}, \eqref{eq:mMkflgcNbFpB}, and Part (E) of (this) Definition \ref{def:bzSncsGnd},  we define  the \emph{combined parameter} $\cop$ of $\PP$ by the following equation.
\begin{equation}
\begin{aligned}
\cop:=\min\bigl\{ &\ffunc(\ncmp)\cdot \thn + \spp(\bSelc_1)
\cr
&+ \spp(\bSelc_2): \bSelc_1,\bSelc_2\in\Comp\PP\text{ and }\bSelc_1\neq\bSelc_2 \bigr\}.
\end{aligned}
\label{eq:kFlTlkNttmNdNk}
\end{equation}

If $\Comp\PP$ contains two or more chains, then we attain the minimum in  \eqref{eq:kFlTlkNttmNdNk} by
selecting $\bSelc_1$ and $\bSelc_2$ as the smallest chain components.
\end{definition}

\begin{figure}[ht] \centerline{ \includegraphics[width=\textwidth]{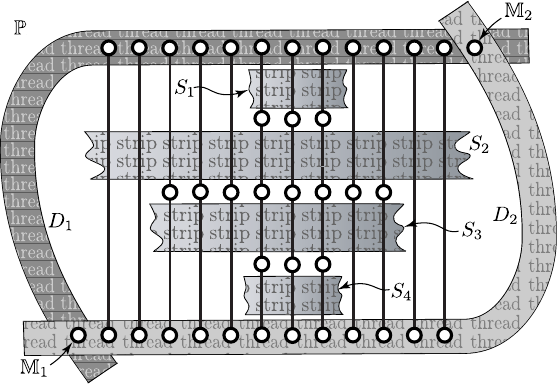}} \caption{A forest $\PP $; see Example \ref{exmpl:jtKgmbRnk}. 
%
 }\label{fig2}
\end{figure}

Based on Definitions \ref{def:fntnS}, \ref{def:fngfnCt}, and \ref{def:bzSncsGnd}(A)--(F), we are in the position to formulate our result on the lattice  $\Pquleq$ of all of quasiorders of $\PP $ that extend $\leq$, which is also denoted by $\mu$.

\begin{theorem}\label{thmmain} 
Let $\PP=(P;\leq)=(P;\mu)$ be a poset of finite component size $\ncs$ such that $|P|$ is an accessible cardinal.
  Then the following three assertions hold.

\textup{(A)} If $\ncmp\geq 2$, 
then $ \Pquleq$ has a complete generating set  $E$ such that 
\begin{equation}
|E|\leq\cop  + \ncedge .
\label{eq:nSmbjvnmrha}
\end{equation}

\textup{(B)} If $\PP $ is a forest and $\ncmp\geq 2$, then $ \Pquleq$ has a complete generating set  $E$ such that 
\begin{equation}
|E|\leq \cop  + \LASp\ncedge .
\label{eq:bnSmgbNDmrhb}
\end{equation}

\textup{(C)} If $\,\PP $ is a finite chain, then
 $\Pquleq$ is $\LASp\ncedge$-generated  but it is not $\left(\LASp\ncedge-1\right)$-generated.
\end{theorem}

As $\PP$ is of finite component size, each of  \eqref{eq:nSmbjvnmrha},  \eqref{eq:bnSmgbNDmrhb}, and $\LASp\ncedge$ in the theorem is an integer. 
For the sake of completeness, we include the following remark; it is less significant than Theorem \ref{thmmain} but provides additional information.

\begin{remark}\label{rem:lsCntFl}
If $\PP$ is a finite connected poset (and so $\ncmp=1$), then $\Pquleq$ is $(\ncedge+\spp(\PP)-1)$ generated, and it is even  $(\LASp\ncedge+\spp(\PP)-1)$ generated if $\PP$ is a finite tree.   
\end{remark}

We are going to prove this remark at the end of Section \ref{sect:proofs}.
For our figures, Theorem \ref{thmmain} asserts  the following.

\begin{example}\label{exmpl:BlTrsst} For $\PP$ drawn in Figure \ref{fig1},  $\Pquleq$  is $22$-generated. Indeed, we have that $\Comp \PP =\set{\TT_1,\dots,\TT_7}$, $\ncmp=7$, $\ncs=6$, and  $\ncedge=5$. As the figure shows, we can choose   
$\Strip :=\set{S_{1},\dots,S_{5}}$, and $\Thread:=\set{D_1,\dots,D_4}$. According to the sentence right after \eqref{eq:kFlTlkNttmNdNk}, we can select $\bSelc_1:=\TT_6$, and  $\bSelc_2:=\TT_7$. 
Then $\thn=2 \cdot 1\cdot 1+2 \cdot 1\cdot 1=4$, 
$\ffunc(\ncmp)=4$, $\spp(\bSelc_1)=1$, $\spp(\bSelc_2)=1$,  
$\cop=4\cdot 4+1+1=18$, $\LASp{\ncedge}=4$, and \eqref{eq:bnSmgbNDmrhb} gives $22$.
\end{example}

\begin{example}\label{exmpl:jtKgmbRnk}
For $\PP$ drawn in Figure \ref{fig2}, 
 $\Quleq(\PP)$ is $12$-generated. The reason is that  $\thn=1+1=2$, $\ffunc(\ncmp)=4$, $\spp(\bSelc_1)=0=\spp(\bSelc_2)$, $\cop=8$, $\ncedge=4=\LASp{\ncedge}$, and  \eqref{eq:bnSmgbNDmrhb} applies.

%
\end{example}

\begin{remark}To explain the connection of the main result and the title of the paper,  
it is reasonable to say that a filter is a \emph{large filter} if it is a principal filter $\filter u$ such that the height of $u$ (that is, the length of $\ideal u$) is small. 
For a set $A$ and $\mu\in\Quo A$,  let $\Theta(\mu)$ denote $\mu\cap\mu^{-1}=\set{(x,y)\in A^2: (x,y)\in\mu\text{ and }(y,x)\in\mu}$. Then $\Theta(\mu)\in\Equ A$, allowing us to consider the \emph{quotient set} $A/\Theta(\mu)$ and \emph{the quotient quasiorder} $\mu/\Theta(\mu)=\set{\bigl(a/\Theta(\mu), b/\Theta(\mu)\bigr): (a,b)\in\mu}$. Clearly, $\mu/\Theta(\mu)$ is a partial order of $A/\Theta(\mu)$. 
There are two paths to establish that the filters  $\pfilter{\Quo A}\mu$ and  $\pfilter{\Quo {A/\Theta(\mu)}} {\mu/\Theta(\mu)}$ are isomorphic lattices. First, this statement is a particular case of (the last sentence of) Cz\'edli and Lenkehegyi \cite[Theorem 1.6]{czglenkh}. Second, a straightforward argument yields the required isomorphism.
Therefore, a large filter $\pfilter{\Quo A}\mu$ of $\Quo A$ is, up to isomorphism, of the form $\Quleq(\PP)$  where $\PP :=\bigl( A/\Theta(\mu); \mu/\Theta(\mu) \bigr)$.
\end{remark}

\section{Proofs}\label{sect:proofs}

\begin{proof}[Proof of Theorem \ref{thmmain}]
While the meet operation in $\Quo \PP $ and  $\Quleq (\PP)$  is the set theoretic intersection, the join in these complete lattices are described with the help of the reflexive and the transitive closure $\rtr$ as follows:
\begin{equation}
\begin{aligned}
\bigvee\set{\rho_j: j\in J}&=\rtr\kern-2pt\left(\bigcup\set{\rho_j: j\in J}\right)\text{ and }\cr
\mbigvee\set{\tau_j: j\in J}&=\rtr\kern-2pt\left(\bigcup\set{\tau_j: j\in J}\right);
\end{aligned}
\label{eq:hNhsntKVnD}
\end{equation}
here and later, $\rtr$ is understood over $P$. 
Since  $\Pquleq$, being a principal filter, is a complete sublattice of $\Quo P$, we will use $\vee$ and $\bigvee$ for the elements of $\Pquleq$. 
For a pair $(x,y)\in P^2$ and a set $\gamma\subseteq P^2$ of such pairs,  let $\quo(x,y)$ and $\quo(\gamma)$ stand for the least quasiorder in $\Quo P$ containing $(x,y)$ and including $\gamma$, respectively. 
Similarly, $\qum(x,y)$ and $\qum(\gamma)$ or $\gamma^+$ are the least members of $\Pquleq$ containing $(x,y)$ and including $\gamma$, respectively. Note that 
\begin{equation}\left.
\parbox{7cm}{
$\gamma^+=\qum(\gamma)=\mu\vee \quo(\gamma)$ and, in particular, $\qum(x,y)=\qum(\set{(x,y)}) =  \mu\vee\quo(x,y)$.}\,\,\right\}
\label{eq:kZdkmGhRnJsz}
\end{equation}
For a subset $X$ of $P$, we define
\[
\RQuleq X:=\set{\rtr(\mu\cup\rho): \rho\in\Quo X}
=\set{\rho^+: \rho\in\Quo X}; 
\]
it is a subposet of $\Pquleq$. 
For a subset $X$ of $P$,  $\nab X$ denotes the ``full relation'' $X^2$ of $X$; then \eqref{eq:kZdkmGhRnJsz} defines $\enab X  \in\Pquleq$.  Observe that 
\begin{equation}
\text{if $\XX\in\Comp \PP $, then 
$\enab X=\mu\cup \nab X$.}
\label{eq:smZnzTsnsdkRc}
\end{equation}

For later use, we formulate the following ``independence principle".
\begin{equation}\left.
\parbox{7.7cm}{If $\XX\in\Comp \PP $,  $p,q\in X$, and $x,y\in P\setminus X$, then $(x,y)\in \qum(p,q)$ if and only if $x\leq y$. An analogous statement holds for $\rho\subseteq X^2$, $x,y\in P\setminus X$, and $(x,y)\in\qum(\rho)$.}\,\,\right\}
\label{eq:gTlsBmpsxsGsTp}
\end{equation}
Since $(p,q)\in\nab T\subseteq \enab T\in\Pquleq$, \eqref{eq:smZnzTsnsdkRc} implies the validity of \eqref{eq:gTlsBmpsxsGsTp}. 

Next, let $\TT\in\Comp\PP$. For $\rho\in \Quleq(\TT)$, $\rho\subseteq T^2\subseteq P^2$, whereby $\qum(\rho)$ makes sense. Hence, we can define a map $\psi_\TT\colon \Quleq(\TT)\to \Quleq(\PP)$ by $\rho\mapsto \qum(\rho)$.   Using \eqref{eq:hNhsntKVnD} or \eqref{eq:gTlsBmpsxsGsTp},  it follows easily that 
\begin{equation}
\psi_\TT\text{ defined above is a complete lattice embedding.}
\label{eq:zJfnNcBrZk}
\end{equation} 
To increase  readability, we partition the rest of the proof into four parts.

\begin{nextpart}
We claim that
\begin{equation}\left. 
\parbox{9.0cm}{if $\emptyset\neq X\subseteq P$
such that $X$ has an at most one-element intersection with each component of  $\PP$, then $\phi_X\colon \Quo X\to \Pquleq$ 
defined by $\rho\mapsto \rtr(\mu\cup\rho)$
is a complete lattice embedding and $\RQuleq X=\phi_X(\Quo X)$; in particular, $\RQuleq X$ is a complete sublattice of $\Pquleq$ and this sublattice is isomorphic to $\Quo X$.}\,\,\right\}
\label{eq:jnmRbgfLFgn}
\end{equation}
For a set $X$, $\Delta_X:=\set{(x,x):x\in X}$ denotes the smallest element of $\Quo X$. To prove \eqref{eq:jnmRbgfLFgn}, consider the functions
\[ 
\begin{aligned} 
\phi'_X&\colon \Quo X \to \Quo P \text{ defined by }\phi'_X(\rho):=\rho\cup \Delta_P\text{ and}
\cr
\phi^\flat_X&\colon \phi'_X(\Quo X)\to \Pquleq \text{ defined by }\phi^\flat_X(\rho):=\rho\vee \mu. 
\end{aligned}
\]
It follows from \eqref{eq:hNhsntKVnD} and $\Delta_P\cup\mu=\mu$ that $\phi_X=\phi_X^\flat \circ\phi'_X$.
Thus, as $\phi'_X$ is clearly an embedding,  
it suffices to show that $\phi^\flat_X$ is also one. 
As $\rtr(\mu\cup\rho)=\mu\vee\rho$ by \eqref{eq:hNhsntKVnD}, it is clear that $\phi^\flat_X$ commutes with joins. 
For $\rho\in\Quo P$, let $\str\rho$ denote the \emph{irreflexive version} $\rho\setminus\Delta_P$ of $\rho$. 
For $\rho \in \phi'_X(\Quo X)$, we have that $\str\rho\subseteq X^2$. Assume that $\rho_j\in \phi'_X(\Quo X)$ for all $j$ in an index set $J$ and $(x,y)\in\bigwedge_{j\in J}\phi^\flat_X(\rho_j)=\bigcap_{j\in J}\phi^\flat_X(\rho_j)$; we need to show that $(x,y)\in
\phi^\flat_X \bigl(\bigcap_{j\in J} \rho_j\bigr)$. Since this is trivial if $(x,y)\in\mu$, we assume that $(x,y)\notin\mu$, that is, $x\nleq y$. 
As $(x,y)\in\bigcap_{j\in J}\phi^\flat_X(\rho_j)$, we have that $(x,y)\in \phi^\flat_X(\rho_j)=\rho_j\vee \mu$ for all $j\in J$. 
By the first line of \eqref{eq:jnmRbgfLFgn}, the transitivity of $\rho_j$, that of $\mu$,  $\rho_j^-\subseteq X^2$,  $x\nleq y$, and \eqref{eq:hNhsntKVnD}, there are elements $z_0=x$, $z_1$, $z_2$, and $z_3=y$ such that $\set{z_1,z_2}\subseteq X$,  $z_1\neq z_2$, 
$(z_0,z_1)\in \mu$, $(z_1,z_2)\in\str\rho_j$, and $(z_2,z_3)\in\mu$. The first two memberships give that $z_1$ is the only element in the intersection of $X$ and the component of $z_0=x$. Similarly, the last two memberships give that $z_2$ is the only element in the intersection of $X$ and the component of $z_3=y$. Thus, $z_1$ and $z_2$ do not depend on $j$, whereby $(z_1,z_2)\in \bigcap_{j\in J} \rho_j$ and we obtain that $(x,y)\in \Bigl(\bigcap_{j\in J} \rho_j\Bigr)\vee \mu= \phi^\flat_X\Bigl(\bigcap_{j\in J} \rho_j\Bigr)$. 
So $\bigcap_{j\in J}\phi^\flat_X(\rho_j)
\subseteq
\phi^\flat_X\Bigl(\bigcap_{j\in J} \rho_j\Bigr)$. This yields that $\phi^\flat_X$ commutes with meets as it is isotone. So $\phi^\flat_X$ is a complete lattice homomorphism.
For $\rho\in \Quo P$, let $\restrict \rho X$ denote the restriction of $\rho$ to $X$. That is, $\restrict \rho X=\rho\cap(X\times X)$. 
The first line of \eqref{eq:jnmRbgfLFgn}
together with \eqref{eq:hNhsntKVnD} imply easily that for $\rho\in \phi_X'(\Quo X)$, we have that 
$\rho=\restrict{\phi^\flat_X(\rho)}X\cup  \Delta_P$. This yields the injectivity of $\phi^\flat_X$ and 
we have shown \eqref{eq:jnmRbgfLFgn}.
\end{nextpart}

For a poset $\XX$, we define $\iEdge \XX:=\set{(x,y): (y,x)\in\Edge \XX}=\set{(x,y):  x,y\in X\text{ and }y\prec x}$. With the help of this notation,  we introduce two crucial concepts for the proof as follows.
By a \emph{strip} of $\PP$, we mean a nonempty subset $S$ of $\iEdge \PP$ such that for all $\XX\in\Comp \PP$, we have that $|S\cap \iEdge  \XX|\leq 1$.
An \emph{upper thread} is a subset $D$ of $\Max \PP$ such that for all $\XX \in\Comp \PP$, we have that $|D\cap \Max \XX|=1$.
Similarly, if $D\subseteq \Min\PP$ such that $|D\cap \Min \XX|=1$ for all $\XX \in\Comp \PP$, then $D$ is a \emph{lower thread}. Upper threads and lower threads are called \emph{threads}. Note the difference: ``$\leq 1$'' for strips but ``$=1$'' for threads. 
Some strips and threads are visualized in Figures \ref{fig1} and \ref{fig2}.

\begin{equation}\left.
\parbox{9.5cm}{The stipulations on $\PP$ allow us to choose a partition $\Strip$ of $\iEdge \PP$ such that $\Strip$ consists of 
$\ncedge$  blocks and each of these blocks is a strip of  $\PP$.
}\,\,\right\}
\label{eq:strpS}
\end{equation}

\begin{nextpart}
This part of the proof is needed  only for Parts (B) and (C) of the theorem. 
We claim that
\begin{equation}\left.
\parbox{7cm}{if $\PP$ is a forest of finite component size, then $\forall (v,u)\in \iEdge \PP$, $\qum(v,u)\nleq \bigvee\set{\qum(y,x)  : (y,x)\in\iEdge \PP\setminus\set{(v,u)}}$.}\right\}
\label{eq:smTsmvllKnrVtd}
\end{equation}
To verify this, suppose the contrary. Then, combining \eqref{eq:hNhsntKVnD} and \eqref{eq:kZdkmGhRnJsz} with $\mu=\rtr(\Edge \PP)$, we can find a \emph{shortest} sequence $z_0=v$, $z_1$,\dots, $z_{n-1}$, $z_n=u$ of elements of $P$ such that for each $s\in\set{1,\dots,n}$,  either $z_{s-1}\prec z_s$,  or $z_{s-1}\succ z_{s}$ but $(z_{s-1},z_{s})\neq (v,u)$. Our sequence, being the a shortest one, is repetition-free. Hence, in particular, none of the $z_1$,\dots, $z_{n-1}$ is $u$ or $v$.
But then $z_0=v$, $z_1$,\dots, $z_{n-1}$, $z_n=u$, $z_{n+1}=z_0$ is a non-singleton circle in the graph of $\PP$, which is a contradiction proving \eqref{eq:smTsmvllKnrVtd}.

For a strip $S\in\Strip$, see \eqref{eq:strpS}, $S$ is a set of reversed edges and so a subset of $P^2$, whence $\qum(S)$ makes sense. For $\TT\in \Comp \PP$,  $|\iEdge \TT\cap S|\leq 1$. Thus, using a shortest sequence like in the previous paragraph and observing that this sequence cannot ``jump'' from a component to another one, we obtain easily that
\begin{equation}\left.
\parbox{7cm}{if $(y,x)\in S\in\Strip$ and $x,y\in \TT\in\Comp \PP$, then $\qum(y,x)=\qum(S)\cap \enab T$.}\,\,\right\}
\label{eq:trZlmSgLxprfTr}
\end{equation}

We claim  that 
\begin{equation}\left.
\parbox{9.5cm}{if $\PP$ is a forest of finite component size, then $\Pquleq$ contains an at most  $\LASp\ncedge$-element subset $G$ such that for each $S\in \Strip$, $\qum(S)$ is the meet of some members of $G$.}\,\,\right\}
\label{eq:hkdPstrmLft}
\end{equation}

For a set $H$, we denote the \emph{power set lattice} of $H$ by $\Pow H$; in this lattice, $\vee$, $\wedge$ and $\leq$ are $\cup$, $\cap$, and $\subseteq$, respectively. 
Since $|\Strip|=\ncedge$, we know from Cz\'edli \cite{czgBoolegen} that $\Pow\Strip$ has a $\LASp\ncedge$-element generating set $G_0$. The members of $G_0$ are subsets of $\Strip$. For each strip $S$, the singleton $\set{S}\in\Pow\Strip$ belongs to the sublattice $[G_0]$ generated by $G_0$. Since $\Pow\Strip$ is distributive, $\set S$ is obtained by applying a disjunctive normal form to appropriate elements of $G_0$. That is, $\set{S}$ is the union of intersections of elements of $\Pow \Strip$. However, $\set{S}$ is join-irreducible (i.e., union-irreducible), whereby no union is needed and $\set S$ is the intersection of some members of $G_0$. So 
\begin{equation}
\set S=\bigcap \set{X: S\in X\in G_0}\,\text{ for every }S\in \Strip.
\label{eq:rgtNmdkNb}
\end{equation}
For $X\in G_0$, we let 
\begin{align}
&\qum(X):=\bigvee\set{\qum(S): S\in X}=\qum\bigl(\bigcup\set{S: S\in X}\bigr)
\label{eq:blnFrvhZja},\\
&\text{and we define }G:=\set{\qum(X): X\in G_0}.
\label{eq:blnFrvhZjb}
\end{align}
We are going to show that $G$ satisfies the requirements of \eqref{eq:hkdPstrmLft}. As $|G_0|=\LASp\ncedge$, we have that $|G|\leq  \LASp\ncedge$, as required.

Next, we are going to show that for  $S$ occurring in \eqref{eq:rgtNmdkNb},
we have that 
\begin{equation}
\qum(S)=\bigwedge\set{\qum(X):X\in G_0\text{ and }S\in X}.
\label{eq:mchGfhlStwmnc}
\end{equation}
The ``$\subseteq$'' part of \eqref{eq:mchGfhlStwmnc} is clear since $S\in X$ implies, by \eqref{eq:blnFrvhZja}, that $\qum(S)\leq \qum(X)$. 
To verify the converse inclusion, assume that $p,q\in P$ such that 
\begin{equation}p\neq q\,\,\text{ and }\,\,(p,q)\in \qum(X)\text{ for all }X\in G_0\text{ containing }S.
\label{eq:mKhcPhfWtcrGKhV}
\end{equation}
For an $X$ such that  $S\in X\in G_0$, as in  \eqref{eq:mKhcPhfWtcrGKhV}, the containment $(p,q)\in \qum(X)$, $\mu=\rtr{(\Edge \PP})$, and  \eqref{eq:hNhsntKVnD}
yield a \emph{shortest} sequence\footnote{For convenience, we write this sequence as a vector.} 
\begin{equation}
\vec r(X)=\big(r_0(X)=p, \, r_1(X),\, \dots,\, r_{m(X)-1}(X),r_{m(X)}(X)=q\bigr)
\label{eq:rshszLndl}
\end{equation}
of elements of $P$ such that for each $s\in\set{1,\dots,m(X)}$,
\begin{itemize}
\item[$(\textup{c}1)$] $r_{s-1}(X)\prec r_s(X)$ or
\item[$(\textup{c}2)$] $(r_{s-1}(X),r_{s}(X))\in \qum(S')$ for some $S'\in X$.
\end{itemize}
By the same reason, the containment in $\qum(S')$ in 
$(\textup{c}2)$ also has an expansion to a sequence of elements of $P$. Thus, assuming that this expansion has already been made
(and resetting the notation), we can replace  $(\textup{c}2)$ by the following condition:
\begin{itemize}
\item[$(\textup{c}2')$]  $(r_{s-1}(X),r_{s}(X))\in  S'\subseteq \iEdge \PP$ for some $S'\in X$.
\label{eq:chnTjkbFMrJlc}
\end{itemize}
So, for each $s\in\set{1,\dots,m(X)}$, either $(\textup{c}1)$  or $(\textup{c}2')$
holds. 
In both cases, $r_{s-1}(X)$ and $r_s(X)$ are the two endpoints of an edge of the (undirected) graph of $\PP$.
Hence, $\vec r(X)$ is a path in the  graph of $\PP$. In a forest, the shortest path connecting two vertices is unique (as otherwise we would get a circle). 
Consequently,  for $X\in G_0$ such that $S \in X$,   
\begin{equation}
\text{the sequence  }\vec r(X)\text{ does not depend on }X.
\label{eq:szWrrGthmnkRn}
\end{equation}
This allows us to write $m$, $\vec r$, and $r_i$ instead of $m(X)$, $\vec r(X)$, and $r_i(X)$, respectively. 
As distinct strips are disjoint, $S'$ in 
$(\textup{c}2')$ is uniquely determined.
Keeping  in mind that $X$ has contained $S$ since \eqref{eq:mKhcPhfWtcrGKhV},  \eqref{eq:szWrrGthmnkRn} and $(\textup{c}2')$ yield that for each $s\in\set{1,\dots,m}$, 
either $(\textup{c}1)$, i.e. $r_{s-1}\prec r_s$,  or
\begin{itemize}
\item[$(\textup{c}3)$]  $(r_{s-1},r_{s})\in  S'\subseteq \iEdge \PP$ for some  $S'$ belonging to all the $X\in G_0$ that contain $S$.
\end{itemize}
It follows from \eqref{eq:rgtNmdkNb} that 
$S'$ in $(\textup{c}3)$ is $S$. 
Hence, if  $(\textup{c}3)$ holds for a subscript $s\in\set{1,\dots,m}$, then $(r_{s-1},r_s)\in S\subseteq \qum(S)$. If $(\textup{c}1)$ holds for this $s$, then $(r_{s-1},r_s)\in \qum(S)$ again since $\mu\subseteq \qum(S)$. Therefore, we obtain that $(p,q)=(r_0,r_m)\in\qum(S)$ by transitivity. 
We have obtained the converse inclusion for \eqref{eq:mchGfhlStwmnc}. Hence, we have proved \eqref{eq:mchGfhlStwmnc} and, thus,  \eqref{eq:hkdPstrmLft}.
\end{nextpart}

\begin{nextpart}
The following ``convexity property'' is valid for any poset $\PP$. 
\begin{equation}\left.
\parbox{9cm}{For $a<b\in \PP$ and $\rho\in\Pquleq$, $(b,a)\in \rho$ if and only if $(y,x)\in\rho$ for all edges $[x,y]$ such that $a\leq x\prec y\leq b$.}\,\,\right\}
\label{eq:nGmtvRnrGDggr}
\end{equation}
The ``if part'' of \eqref{eq:nGmtvRnrGDggr} is clear by transitivity. To see the ``only if part'', assume that $(b,a)\in\rho$ and $[x,y]$ is an edge in the interval $[a,b]$. Then $(y,b)\in\mu\subseteq \rho$, $(b,a)\in\rho$, and $(a,x)\in\mu\subseteq \rho$ imply $(y,x)\in\rho$ by transitivity, proving \eqref{eq:nGmtvRnrGDggr}.

Now we are in the position to prove Part (C) of the theorem. So let $\PP$ be a chain of length $\ncedge$ for a while. 
By Kim, Kwon, and Lee \cite{kimatal}, $\Pquleq$ is a Boolean lattice with $\ncedge$ atoms.
Alternatively, since all the necessary tools occur in the present paper, we can argue shortly as follows. For each member $\rho$ of $\Pquleq$,   \eqref{eq:nGmtvRnrGDggr} gives that 
\[
\rho=\bigvee\{\qum(y,x): (y,x)\in\iEdge \PP\cap \rho\}.
\] 
By \eqref{eq:smTsmvllKnrVtd}, none of the joinands can be removed, and it follows that  $\Pquleq$ is (order) isomorphic to the powerset 
lattice of $\iEdge \PP$, whereby it is
isomorphic to the Boolean lattice with $\ncedge$ atoms. Therefore, Part (C) of Theorem \ref{thmmain} follows from Cz\'edli \cite{czgBoolegen}.
\end{nextpart}

\begin{nextpart}
Now we turn our attention to Parts (A) and (B) of the theorem. 
For an antichain, the theorem asserts the same as Lemma \ref{lemma:czkLn}. Thus, in the rest of the proof, we assume that $\PP$ is not an antichain. 
For Part (B), we choose $G$ according to \eqref{eq:hkdPstrmLft}. For Part (A), let $G:=\Strip$; see \eqref{eq:strpS}. It is trivial for Part (A) and it follows from  \eqref{eq:hkdPstrmLft} for Part (B) that 
\begin{equation}\left.
\parbox{9.2cm}
{$|G|=\ncedge$ for Part (A), $|G|=\LASp\ncedge$ for Part (B), and (for both of these two parts of the theorem) for each $S\in\Strip$, $\qum(S)$ is the meet of some members of $G$.   
}\,\,\right\}
\label{eq:mGmCrsWthW}
\end{equation}

Next, we choose a set $\Thread$ of threads of $\PP$; it is not unique in general. Choose $\bSelc_1,\bSelc_2\in\Comp\PP$ such that
\begin{equation}
\bSelc_1\text{ and }\bSelc_2\text{ witness the minimum in \eqref{eq:kFlTlkNttmNdNk}}. 
\label{eq:szNsTlkZJqkl}
\end{equation}
Although $\bSelc_1$ and $\bSelc_2$ are not uniquely determined in general, we fix an $\bSelc_1$ and an $\bSelc_2$ satisfying \eqref{eq:szNsTlkZJqkl}. If $\ncmp\geq 3$, then choose a set $T^{(\text{up})}$ of upper threads of $\PPdnul$ such that 
\[
|T^{(\text{up})}|=\max\set{|\Max {\XX}|: \XX\in \Comp \PPdnul}
\]
and each $x\in\Max\PPdnul$ belongs to a thread  $D\in T^{(\text{up})}$. 
Dually, if $\ncmp\geq 3$, then we choose a set $T^{(\text{lo})}$ of lower threads of $\PPdnul$ such that 
\[
|T^{(\text{lo})}|=\max\set{|\Min {\XX}|: \XX\in \Comp \PPdnul}
\]
and  $\Min\PPdnul=\bigcup\set{D: D\in T^{(\text{lo})}}$.
If $\ncmp=2$, then we consider $\emptyset$ a thread of the ``empty poset'' $\PPdnul$ and we let  $T^{(\text{up})}:=\set\emptyset$ and $T^{(\text{lo})}:=\set\emptyset$.
By adding an element of $\Max{\bSelc_1}$ and that of $\Max{\bSelc_2}$ to a thread $D\in T^{(\text{up})}$, we can extend $D$ to an upper thread of $\PP$; note that $D$ has $|\Max{\bSelc_1}|\cdot |\Max{\bSelc_2}|$ such extensions. Taking all $D\in T^{(\text{up})}$ into account, these extensions form a set that we denote by $T^{(\text{up})}_{\text{ext}}$. We define $T^{(\text{lo})}_{\text{ext}}$ dually, and we let 
$\Thread:= T^{(\text{up})}_{\text{ext}}\cup T^{(\text{lo})}_{\text{ext}}$. As $\PP$ is not an antichain, $T^{(\text{up})}_{\text{ext}}\cap T^{(\text{lo})}_{\text{ext}}=\emptyset$. 
By \eqref{eq:kNgbRkssRjb} and \eqref{eq:mMkflgcNbFpB},  $|\Thread|=\thn$.

Clearly, for all $\TT\in\Comp\PP$, $i\in\set{1,2}$,
$x_\ast\in\Min\TT$, $y_\ast\in\Min{\bSelc_i}$,
$x^\ast\in\Max\TT$, and $y^\ast\in\Max{\bSelc_i}$, there exist $D',D''\in\Thread$ such that 
\begin{equation}
\set{x_\ast,y_\ast}\subseteq D'\text{ and }\set{x^\ast,y^\ast}\subseteq D''.
\label{eq:dzfSlfmDkKsrbk}
\end{equation}

For $D\in\Thread$, $|D|=\ncmp$ by definitions. 
So Lemma \ref{lemma:czkLn} and Definition \ref{def:fngfnCt}
allow us to pick an $\ffunc(\ncmp)$-element generating set $H_D^{(0)}$ of $\Quo D$. By \eqref{eq:jnmRbgfLFgn},  $\RQuleq D$, which is a complete sublattice of $\Pquleq$, also has an  $f(\ncmp)$-element generating set $H_D:=\phi_D(H_D^{(0)})$. Let
\begin{equation}
H:=\bigcup_{D\in\Thread} H_D\text{; note that }|H|=\thn\cdot \ffunc(\ncmp).
\label{eq:hnJkzMzPdFk}
\end{equation}

Next, let $i\in \set{1,2}$ and keep  \eqref{eq:szNsTlkZJqkl} in mind. If $\spp(\bSelc_i)=0$ (that is, if $|\Selc_i|=1$), then $F^{(0)}(\bSelc_i):=\emptyset$. If  $\bSelc_i$ is a finite chain with more than one element, then let $F^{(0)}(\bSelc_i):= \set{\nab{\Selc_i}}$. 
If $\bSelc_i$ is not a chain, then let $F^{(0)}(\bSelc_i)$ be the union of  $\set{\nab{\Selc_i}}$ and the set $Y(\bSelc_i)$; see \eqref{eq:ZlsJdlFRtXgmnRt}.
In all cases, $F^{(0)}(\bSelc_i)\subseteq  \Quleq(\bSelc_i)$.  Note that 
\begin{equation}
\text{for $i\in\set{1,2}$, \ $|F^{(0)}(\bSelc_i)| = \spp(\bSelc_i)$.}
\label{eq:mchNnzgRBltnFdWr}
\end{equation}
Now we pass from $\bSelc_i$ to $\PP$. For $i\in\set{1,2}$, let $\psi_i$ stand for $\psi_{\bSelc_i}$; see \eqref{eq:zJfnNcBrZk}.
Let $\SSelc_i:=\psi_i\bigl(\Quleq(\bSelc_i)\bigr)$ and 
$F(\SSelc_i):=\psi_i(F^{(0)}(\bSelc_i))$; the superscript of $\SSelc_i$ is just a reminder to ``\pmb{\tbf s}ub\pmb{\tbf l}attice''. 
Note the difference: $\bSelc_i$ is in $\Comp \PP$ and so it consists of some elements of $P$ while  $\SSelc_i$ is a sublattice of $\Pquleq$ and consists of some quasiorders of $P$.
As a subset of $\SSelc_i$, $F(\SSelc_i)$ is also a subset of $\Pquleq$.
It is clear by \eqref{eq:ZlsJdlFRtXgmnRt} and \eqref{eq:mchNnzgRBltnFdWr}  that for $i\in\set{1,2}$
\begin{gather}
F(\SSelc_i)\cup\iEdge{\bSelc_i} \text{ generates }
\SSelc_i, \quad \enab{\Selc_i} \in F(\SSelc_i),
\label{eq:lCrgMszFlgBa}
\\
\text{and \ }|F(\SSelc_i)|= \spp(\bSelc_i).
\label{eq:lxMszFlPblgB}
\end{gather}
Now \eqref{eq:mGmCrsWthW}, \eqref{eq:hnJkzMzPdFk}, and \eqref{eq:lCrgMszFlgBa} allow us to define
\begin{equation}
E:=F(\SSelc_1)\cup F(\SSelc_2)\cup G\cup H .
\label{eq:vgHjvTrVlC}
\end{equation}
It follows from \eqref{eq:mGmCrsWthW}, \eqref{eq:hnJkzMzPdFk}, and \eqref{eq:lxMszFlPblgB} that, depending on whether we deal with Part (A) or Part (B) of the theorem, $|E|$ satisfies \eqref{eq:nSmbjvnmrha} or \eqref{eq:bnSmgbNDmrhb}, respectively. 
Thus, it suffices to show
that $E$ generates $\Pquleq$.
Let $K$ be a complete sublattice of $\Pquleq$  such that $E\subseteq K$. To complete the proof, we need to show that $\Pquleq\subseteq K$. It follows from $H_D\subseteq H\subseteq E$ that
\begin{equation}
\text{for all }D\in\Thread, \quad \RQuleq D\subseteq K.
\end{equation}

Next, we claim that
\begin{equation}
\text{for any $i\in\set{1,2}$ and $(y,x)\in\iEdge{\bSelc_i}$, \ $\qum(y,x)\in K$.}
\label{eq:wDgkNwrtsflWr}
\end{equation}
To see this, take the unique $S\in\Strip$ that contains the reversed edge $(y,x)$ of $\bSelc_i$. Since $G\subseteq E$, \eqref{eq:mGmCrsWthW} yields that 
$\qum(S)$ is in $K$. So is $\enab {\Selc_i}$ by 
$F(\SSelc_i)\subseteq E$ and \eqref{eq:lCrgMszFlgBa}. Hence \eqref{eq:trZlmSgLxprfTr} gives that
$\qum(y,x)=\qum(S)\wedge \enab {\Selc_i}\in K$, proving 
\eqref{eq:wDgkNwrtsflWr}. Now it follows from 
\nothing{the ``independence principle'' \eqref{eq:gTlsBmpsxsGsTp},}
\eqref{eq:wDgkNwrtsflWr}, $F(\SSelc_i)\subseteq E$, and \eqref{eq:lCrgMszFlgBa} that for all $i\in\set{1,2}$,
\begin{equation}
\text{$\SSelc_i\subseteq K$, that is, for any $x,y\in\Selc_i$, \ $\qum(x,y)\in K$.}
\label{eq:knTcldBspThswN}
\end{equation}

Our next task (to be  completed in the two lines after \eqref{eq:kHwkSrdfksznT})
is to show that 
\begin{equation}\left.
\parbox{6.5cm}{for any $x_\ast\in\Min\PP$ and $y^\ast\in\Max\PP$, if $x_\ast\leq y^\ast$, then $\qum(y^\ast,x_\ast)\in K$.}\,\,\right\}
\label{eq:brSprhjMngzsLfnD}
\end{equation}
To show this, let $\TT$ be the unique component of $\PP$ that contains $x_\ast$ and $y^\ast$. We can assume that $x_\ast <  y^\ast$.  By \eqref{eq:knTcldBspThswN}, we can also assume that $\TT\notin\set{\bSelc_1,\bSelc_2}$. For $i\in\set{1,2}$, pick an element $x_{\ast i}\in\Min{\bSelc_i}$ and an element $y^\ast_i\in\Max{\bSelc_i}$ such that $x_{\ast i}\leq y^\ast_i$. 
In virtue of \eqref{eq:dzfSlfmDkKsrbk}, we can select   $D'_i,D''_i\in\Thread$ such that $\set{x_\ast, x_{\ast i}}  \subseteq D'_i$ and $\set{x^\ast, x^\ast_i}  \subseteq D''_i$.
By the choice of $H_D$ (with $D=D'_i$ or $D=D''_i$), \eqref{eq:jnmRbgfLFgn}, and \eqref{eq:hnJkzMzPdFk},
\begin{equation}
\qum(y^\ast,y^\ast_i)\in K\text{ and }\qum(x_{\ast i},x_\ast)\in K\text{ for }i\in\set{1,2}.
\label{eq:rmMkMjdlkMgndj}
\end{equation}
We know from \eqref{eq:knTcldBspThswN} that
\begin{equation}
\text{for }i\in\set{1,2},\text{ }\qum(y^\ast_i,x_{\ast i})\in K.
\label{eq:knPrvClRhbT}
\end{equation}
We claim that 
\begin{equation}
\spech(y^\ast,x_\ast):={} \bigwedge_{i=1}^2
\bigl( \qum(y^\ast,y^\ast_i) \vee \qum(y^\ast_i,x_{\ast i}) \vee \qum(x_{\ast i},x_\ast)
\bigr)  \in K
\label{eq:kbjkkRlsgnvSb}
\end{equation}
and
\begin{equation}
\qum(y^\ast,x_\ast)\leq \spech(y^\ast,x_\ast)   \leq \enab T .
\label{eq:kbjkkRlsgnvSa}
\end{equation}
%
Observe that \eqref{eq:kbjkkRlsgnvSb} is clear by \eqref{eq:rmMkMjdlkMgndj} and \eqref{eq:knPrvClRhbT}. The first inequality in \eqref{eq:kbjkkRlsgnvSa} is trivial by transitivity. To show the second inequality in \eqref{eq:kbjkkRlsgnvSa}, assume that $(u,v)\in P^2$ belongs to the meet in \eqref{eq:kbjkkRlsgnvSb}, that is, it belongs to $\spech(y^\ast,x_\ast)$. We also assume that $u\nleq v$ since otherwise $(u,v)\in \enab T$ is trivial. Let $i\in\set{1,2}$. Using the description of joins, see \eqref{eq:hNhsntKVnD}, \eqref{eq:kZdkmGhRnJsz}, and \eqref{eq:smZnzTsnsdkRc},  we can take a \emph{shortest} sequence $\vec w: w_0=u, w_1, w_2,\dots, w_{t-1}, w_t=v$ of pairwise different elements of $P$ such that for each $j\in\set{1,\dots,t}$, one of the following four alternatives holds: 
\begin{itemize}
\item[$(\textup{d}1)$] $w_{j-1}<w_j$ \ (and so $w_{j-1}$ and $w_j$ belong to the same component of $\PP$);
\item[$(\textup{d}2)$] $(w_{j-1},w_j)=(y^\ast,y^\ast_i)$;
\item[$(\textup{d}3)$]  $(w_{j-1},w_j)=(y^\ast_i,x_{\ast i})$;
\item[$(\textup{d}4)$]  $(w_{j-1},w_j)=(x_{\ast i},x_\ast)$.
\end{itemize}
Since $u\nleq v$, at least one of $(\textup{d}2)$, $(\textup{d}3)$, and $(\textup{d}4)$ holds some $j\in\set{1,\dots,t}$. Hence, at least one of the elements $w_j$ belongs to $T\cup\Selc_i$.  This fact and  the parenthesized comment in $(\textup{d}1)$ yield that  all the  $w_j$ belong to $T\cup\Selc_i$. In particular, $\set{u,v}=\set{w_0,w_t}\subseteq T\cup\Selc_i$. As opposed to the elements $w_j$ for $j\in\set{1,\dots,t-1}$, we know that $\set{u,v}$ does not depend on $i\in\set{1,2}$. Hence, $\set{u,v}\subseteq ( T\cup\Selc_{i})\cap ( T\cup\Selc_{3-i})=T\cup(\Selc_i\cap \Selc_{3-i})=T\cup\emptyset =
T$. Thus, $(u,v)\in\enab T$, proving  the second inequality of \eqref{eq:kbjkkRlsgnvSa} and \eqref{eq:kbjkkRlsgnvSa}  itself.

Based on \eqref{eq:kbjkkRlsgnvSa} and \eqref{eq:kbjkkRlsgnvSb}, now we show that
\begin{equation}
\text{for any }u,v\in P,\text{ if }u<v,\text{ then }\qum(v,u)\in K.
\label{eq:zZjlbRpnGpbpl}
\end{equation}
%
%
First, we deal with the particular case where $(v,u)\in\iEdge \PP$. Let $S\in\Strip$ be the unique strip that contains $(v,u)$.  Since $G\subseteq E$, we know from \eqref{eq:mGmCrsWthW} that $\qum(S)$ is in $K$. Pick $u_\ast\in \Min\PP$ and $v^\ast\in\Max \PP$ such that $u_\ast\leq u\prec v\leq v^\ast$ and let $\TT$ be the unique component of $\PP$ that contains 
these four elements. 
%
Combining \eqref{eq:trZlmSgLxprfTr}, \eqref{eq:nGmtvRnrGDggr}, and  \eqref{eq:kbjkkRlsgnvSa}, we have that
\begin{align*}
\qum(v,u)
&=\qum(S)\wedge\qum(v,u)
\overset{\eqref{eq:nGmtvRnrGDggr}}\leq 
\qum(S)\wedge \qum(v^\ast, u_\ast)
\cr
&
\overset{\eqref{eq:kbjkkRlsgnvSa}}
\leq
\qum(S)\wedge \spech(v^\ast,u_\ast)
\overset{\eqref{eq:kbjkkRlsgnvSa}}\leq \qum(S) \wedge \enab T \overset{\eqref{eq:trZlmSgLxprfTr} }= \qum(v, u).
\end{align*}
These inequalities and  \eqref{eq:kbjkkRlsgnvSb} yield that 
$\qum(v, u)
=
\qum(S)\wedge \spech(v^\ast,u_\ast)\in K$. That is, the particular case $u\prec v$ of \eqref{eq:zZjlbRpnGpbpl} holds.
It follows from transitivity and \eqref{eq:nGmtvRnrGDggr} that
\begin{equation}
\text{for any $a<b\in \PP$, }\,\,\,
\qum(b,a)=\bigvee\set{\qum(y,x): a\leq x\prec y\leq b}.
\label{eq:kHwkSrdfksznT}
\end{equation}
Therefore, the particular case implies the general case, and we have shown \eqref{eq:zZjlbRpnGpbpl}. 
We have also shown \eqref{eq:brSprhjMngzsLfnD} since it is a particular case of \eqref{eq:zZjlbRpnGpbpl}.

Now we are ready to show that 
\begin{equation}
\text{for each   $\TT=(T;\mu_T)\in\Comp \PP$, \quad  $\enab T\in K$.}
\label{eq:brnszRrtn}
\end{equation}
Here, of course, $\mu_T$ stands for the restriction $\restrict \mu T$ of $\mu$ to $T$. We can assume that $|T|>1$. In virtue of \eqref{eq:zZjlbRpnGpbpl}, it suffices to show that $\enab T=\bigvee\set{\qum(y,x): (y,x)\in\iEdge T} $, that is, $\enab T=\qum(\iEdge T)$. To see the ``$\leq$'' part, assume that $(x,y)\in\enab T$. As $\TT$ is a \emph{connected} component, there are elements $z_0=x, z_1,\dots, z_{t-1}, z_t=y$ in $T$ such
$(z_{j-1},z_j)\in\Edge \TT\cup\iEdge \TT$ for every $j\in\set{1,\dots,t}$. 
Regardless whether $(z_{j-1},z_j)$ is in $\Edge \TT$ or it is in $\iEdge \TT$, we have that $(z_{j-1},z_j)\in\qum(z_{j-1},z_j)\subseteq \qum(\iEdge \TT)$. Thus the  ``$\leq$'' 
 part of the required equality follows by transitivity. As the converse inequality is trivial by \eqref{eq:gTlsBmpsxsGsTp}, we conclude \eqref{eq:brnszRrtn}.

The following easy assertion will be useful.
\begin{equation}
\text{For any }
a,b\in P, 
\text{ we have that }
\qum(a,b)=\mu\cup(\ideal a\times\filter b).  
\label{eq:ddTldmwHnc}
\end{equation}
The ``$\geq$'' part is clear by transitivity. To show the converse inequality, assume that $(p,q)\in \qum(a,b)\setminus \mu$. By \eqref{eq:hNhsntKVnD} and \eqref{eq:kZdkmGhRnJsz}, there exists a sequence $x_0=p, x_1,\dots, x_{t-1}, x_t=q$ of elements in $P$ such that, for each $i\in\set{0,\dots, t-1}$,   $(x_i,x_{i+1})\in \mu$ or  $(x_i,x_{i+1})=(a,b)$. As $p\nleq q$,  there is at least one $i$ such that $(x_i,x_{i+1})=(a,b)$. The least such $i$ shows that $p=x_0\leq x_1\dots\leq x_i=a$, whence $p\in\ideal a$, while the largest such $i$ yields that $b=x_{i+1}\leq\dots\leq x_t=q$, whereby $q\in\filter b$. Hence, $(p,q)\in\ideal a\times \filter b \subseteq \mu\cup(\ideal a\times\filter b)$, proving \eqref{eq:ddTldmwHnc}.

Next, we claim that 
\begin{equation}
\parbox{10cm}{for any $\TT_1,\TT_2\in\Comp\PP$, if 
$\TT_1\neq\TT_2$,  $\set{\bSelc_1,\bSelc_2}\cap\set{\TT_1,\TT_2}\neq\emptyset$, $a\in T_1$ and $b\in T_2$, then  $\qum(a,b)\in K$.}
\label{eq:bwmkhPDhgWrtcndSrt}
\end{equation}
To show this, pick $a_\ast\in\Min{\TT_1}$, $a^\ast\in \Max{\TT_1}$, $b_\ast \in\Min{\TT_2}$,  and $ b^\ast\in \Max{\TT_2}$ such that $a_\ast\leq a\leq a^\ast$ and $b_\ast\leq b\leq b^\ast$. 
Let 
\begin{equation*}
\bup:= \qum(a^\ast,b^\ast) \vee
\qum(b^\ast,b)\text{ and }
\gdn:= \qum(a, a_\ast) \vee\qum(a_\ast,b_\ast). 
\end{equation*}
We claim that 
\begin{equation}
\qum(a,b)=\bup\wedge \gdn\in K.  \label{eq:fFrdmNgRrxc}
\end{equation}
Using the assumption $\set{\bSelc_1,\bSelc_2}\cap\set{\TT_1,\TT_2}\neq\emptyset$,  \eqref{eq:dzfSlfmDkKsrbk} allows us to  select $D',D''\in\Thread$ such that $\set{a_\ast,b_\ast}\subseteq D'$ and $\set{a^\ast,b^\ast}\subseteq D''$. These two inclusions together with $H\subseteq E$ imply that $\qum(a_\ast,b_\ast)$ and $\qum(a^\ast,b^\ast)$ are in $K$. Combining this with \eqref{eq:zZjlbRpnGpbpl}, we obtain that $\bup\wedge \gdn\in K$.

Hence, to prove \eqref{eq:fFrdmNgRrxc} and, consequently, \eqref{eq:bwmkhPDhgWrtcndSrt}, it suffices to prove the equality in \eqref{eq:fFrdmNgRrxc}.
Actually, it suffices to show that $\qum(a,b) \geq \bup\wedge \gdn$, as the converse inequality is trivial. To do so, assume that $(p,q)\in \bup\wedge \gdn$. Pick a shortest sequence $\vec x: x_0=p$, $x_1,\dots,x_{h}=q$ 
and a shortest sequence $\vec y: y_0=p$, $y_1,\dots,$ $y_k=q$
of elements of $P$ witnessing that $(p,q)\in \bup= \mu\vee \quo(a^\ast,b^\ast) \vee
\quo(b^\ast,b)$ and $(p,q)\in\gdn=\mu\vee \quo(a,a_\ast) \vee\quo(a_\ast,b_\ast)$, respectively, according \eqref{eq:hNhsntKVnD}. So, for example, for every $j\in\set{1,\dots,h}$, exactly one of $x_{j-1}<x_j$, $(x_{j-1},x_j)=(a^\ast,b^\ast)$, and $(x_{j-1},x_j)=(b^\ast,b)$ holds. 
First, assume that $q\in T_1$ and $(p,q)\in\bup\wedge\gdn\subseteq \bup$. If an element of the sequence $\vec x$ is in $T_2$, then all the \emph{subsequent} elements of $\vec x$ are in $T_2$, contradicting that $q\in T_1$. Hence, no element of $\vec x$ is in $T_2$. So, for all $j\in\set{1,\dots,h}$, \ $(x_{j-1},x_j)$ is neither $(a^\ast,b^\ast)$ nor $(b^\ast,b)$. Hence, $x_{j-1}<x_j$ for all $j\in\set{1,\dots,h}$, and so $(p,q)=(x_0,x_h)\in\mu\subseteq \qum(a,b)$. 
Let us summarize:
\begin{equation}
\text{if $q\in T_1$, then $(p,q)\in\bup$ implies that $(p,q)\in \mu\subseteq \qum(a,b)$.}
\label{eq:mnlPhghrS}
\end{equation}
Second, observe that if an element of  $\vec y$ belongs to $T_1$, then all the \emph{preceding} elements of $\vec y$ do so.
Then, similarly to the argument for \eqref{eq:mnlPhghrS}, we obtain that 
\begin{equation}
\text{if $p\in T_2$, then $(p,q)$ in $\gdn$ implies that $(p,q)\in \mu\subseteq \qum(a,b)$.}
\label{eq:mknhkSkLv}
\end{equation}
By the ``$\in\mu$'' part of \eqref{eq:mnlPhghrS} and \eqref{eq:mknhkSkLv}, each of the sequences $\vec x$ and $\vec y$ implies that 
\begin{equation}
(p,q)\in T_2\times T_1\text{ is impossible.}
\label{eq:mhlfdksRg}
\end{equation} 
We can assume that $p\nleq q$ as otherwise $(p,q)\in\qum(a,b)$ is obvious. But then each of the sequences $\vec x$ and $\vec y$ implies easily that $\set{p,q}\subseteq T_1\cup T_2$. Thus, it follows from \eqref{eq:mnlPhghrS}, \eqref{eq:mknhkSkLv}, and \eqref{eq:mhlfdksRg} that there is only one case left. Namely, from now on in the proof, we assume that $p\in T_1$ and $q\in T_2$. In the sequence $\vec x$, which begins in $T_1$ and terminates in $T_2$, let $x_u$ be the first element that is in $T_2$. As we mentioned in the argument for \eqref{eq:mnlPhghrS}, the subsequent elements $x_{u+1}$, $x_{u+2}$, \dots, $x_h$ are all in $T_2$. Clearly, $x_u=b^\ast\geq b$. Using that, for $j\in\set{u,\dots, h-1}$, either
$(x_j,x_{j+1})=(b^\ast,b)$ (and so $ x_{j+1}=b\geq b$) or $(x_j,x_{j+1})\in\mu$ (and so $x_j\geq b\then x_{j+1}\geq b$), it follows by induction  on $j$ that $q=x_h\geq b$, that is, $q\in \filter b$. Similarly, letting $y_v$ denote the last member of $\vec y$ that belongs to $T_1$, the elements $y_v, y_{v-1}, \dots, y_0=p$ are all in $T_1$. Then $y_v=a_\ast\in\ideal a$.
Using that each of
the pairs $(y_{j-1},y_j)$, $j\in\set{v,v-1,\dots,1}$, is in $\mu$ or is of the form $(a,a_\ast)$, a trivial induction on $j$ ``downwards'' gives that $p=y_0\in\ideal a$. Hence $(a,b)\in\ideal a\times \filter b$. So $(a,b)\in \qum(a,b)$ by \eqref{eq:ddTldmwHnc}.
We have proved \eqref{eq:fFrdmNgRrxc} and \eqref{eq:bwmkhPDhgWrtcndSrt}.

Next, we claim that 
\begin{equation}
\text{for all }a,b\in P,\text{ we have that }  \qum(a,b)\in K.
\label{eq:nmZkskjLbrxNkYx}
\end{equation}
To prove \eqref{eq:nmZkskjLbrxNkYx}, note that 
if $a\leq b$, then $\qum(a,b)=\mu\in K$ is clear by, say, \eqref{eq:knTcldBspThswN}. If $b<a$, then 
$\qum(a,b)\in K$ follows from \eqref{eq:zZjlbRpnGpbpl}.
If $\set{a,b}\subseteq \Selc_i$ for some $i\in\set{1,2}$, then 
\eqref{eq:lCrgMszFlgBa}, \eqref{eq:vgHjvTrVlC}, and \eqref{eq:zZjlbRpnGpbpl} yield that $\qum(a,b)\in K$. 
If there is an $i\in\set{1,2}$ such that $|\set{a,b}\cap\Selc_i|=1$, then \eqref{eq:bwmkhPDhgWrtcndSrt} implies that $\qum(a,b)\in K$. After these considerations, we can assume that $a\parallel b$ and  $\set{a,b}\subseteq \Pdnul=P\setminus(\Selc_1\cup\Selc_2)$.  For $i\in\set{1,2}$, pick an element $c_i\in \Selc_i$. It suffices to show that
\begin{equation}
\qum(a,b)=\bigl(\qum(a,c_1)\vee\qum(c_1,b) \bigr) \wedge \bigl(\qum(a,c_2)\vee\qum(c_2,b) \bigr),
\label{eq:pTftmgkZbtLdj}
\end{equation}
as the term on the right is in $K$ by \eqref{eq:bwmkhPDhgWrtcndSrt}. The ``$\leq$'' in place of the equality sign in \eqref{eq:pTftmgkZbtLdj} is clear by transitivity. For $i\in\set{1,2}$, \eqref{eq:ddTldmwHnc} yields that 
\begin{equation}
\qum(a,c_i)\subseteq \mu \cup \bigl(\ideal a\times \Selc_i\bigr)
\text{ \ and \ }
\qum(c_i,b)\subseteq \mu  \cup \bigl( \Selc_i\times \filter b\bigr).
\label{eq_jTsZnDkhBkMb}
\end{equation} 
Letting 
\[\eta_i:=(\Selc_i\times\Selc_i)\cup (\Pdnul\times \Selc_i)\cup(\Selc_i\times \Pdnul),
\] 
\eqref{eq:hNhsntKVnD} and \eqref{eq_jTsZnDkhBkMb} imply that
\begin{equation}
\qum(a,c_i)\vee\qum(c_i,b)\subseteq
\mu\cup(\ideal a\times\filter b)\cup \eta_i.
\label{eq:grjsKncnTrlTrkh}
\end{equation}
As $\Pdnul$, $\Selc_1$, and $\Selc_2$ are pairwise disjoint,  \eqref{eq:grjsKncnTrlTrkh}  implies that
\begin{align*}
\bigl(\qum(&a,c_1)\vee\qum(c_1,b) \bigr) \wedge \bigl(\qum(a,c_2)\vee\qum(c_2,b) \bigr)\cr 
\subseteq &\bigl( {\mu\cup(\ideal a\times\filter b)}  \mathrel{\cup} \eta_1\bigr) \cap  \bigl( {\mu\cup(\ideal a\times\filter b)}  \mathrel{\cup} \eta_2\bigr)
\cr = &{\mu\cup(\ideal a\times\filter b)} \mathrel{\cup} (\eta_1\cap\eta_2)
\cr = &{\mu\cup(\ideal a\times\filter b)} \mathrel{\cup} {\emptyset} 
\overset{\eqref{eq:ddTldmwHnc}}= \qum(a,b).
\end{align*}
This shows the ``$\geq$'' inequality and the whole \eqref{eq:pTftmgkZbtLdj}. Thus, \eqref{eq:nmZkskjLbrxNkYx} holds. 

Finally, each element of $\Pquleq$ is the join of (possibly infinitely many) elements of the form $\qum(a,b)$ where $a,b\in P$. Therefore, \eqref{eq:nmZkskjLbrxNkYx} implies that  $\Pquleq\subseteq K$, completing the proof of Theorem \ref{thmmain}.  
\end{nextpart}
\end{proof}

\begin{proof}[Proof of Remark \ref{rem:lsCntFl}]
We are going to use some parts of the previous proof.
As $\ncmp=1$, each strip is a singleton and so we can consider it a reversed edge. Apart from this insignificant change, define $G$ in the same way as we did right before \eqref{eq:mGmCrsWthW}.  Let $\bSelc_1=\PP$ and take $Y(\bSelc_1)$ from Definition \ref{def:bzSncsGnd}(E). Changing \eqref{eq:vgHjvTrVlC} to $E:=F(\SSelc_1)\cup G$, it follows from  \eqref{eq:ZlsJdlFRtXgmnRt} and  \eqref{eq:mGmCrsWthW} that $E$ generates $\Pquleq$. As $|E|$ is what Remark \ref{rem:lsCntFl} requires, the proof is complete.
\end{proof}

\section{Some notes on Theorem \ref{thmmain}}\label{sect:correm}
\begin{remark}\label{rem:nGzkVmW} If $\PP$ is an antichain such that $\ncmp\geq 5$ 
is an accessible cardinal, then Theorem \ref{thmmain}(B) says that $\Pquleq$ is $4$-generated. 
\end{remark}

This means that Theorem \ref{thmmain} implies the first half Lemma \ref{lemma:czkLn} for sets with at least five elements.
However, note that the proof of Theorem \ref{thmmain} uses  Lemma \ref{lemma:czkLn}.

Let $\{\PP_i=(P_i;\leq_i)=(P_i;\mu_i)$ : \ $i\in I\}$ be a set of posets that are assumed to be pairwise disjoint. (If they are not disjoint, then we have to replace them by appropriate isomorphic copies.)  The \emph{cardinal sum} \  $\Csum_{i\in I} \PP_i$ of these posets is $(\bigcup_{i\in I} P_i; \bigcup_{i\in I} \mu_i)$. E.g., for $\PP$ in Theorem \ref{thmmain}, $\PP=\Csum_{\TT\in\Comp \PP}\TT$. If $I=\set{1,2}$, then we prefer to write $\PP_1 \cplus \PP_2$.
The following statement is a straightforward consequence  of Theorem \ref{thmmain}.

\begin{corollary}\label{cor:zsbPl}
Assume that $\kappa\geq 5$ is an accessible cardinal. Let $\bbA_2$ denote the $2$-element antichain. If $\PP$ is the cardinal sum of $\kappa$ many chains and the supremum of the lengths
of these chains is $\ncedge\in\Nplu$, 
then $\Pquleq$ and $\Quleq(\PP\cplus \bbA_2)$, as complete lattices, are $\bigl(10+\LASp\ncedge\bigr)$-generated and $\bigl(8+\LASp\ncedge\bigr)$-generated, respectively.
\end{corollary}

The \emph{Y-poset} is the 4-element poset $\set{0,a,b,c}$ such that $0\prec a\prec b$, $a\prec c$, but $b$ and $c$ are incomparable. 
Based on Theorem \ref{thmmain} and Corollary \ref{cor:zsbPl}, we present some examples.

\begin{example}\label{ex:mPkrDsz}
Let $\kappa\geq 5$, let $\PP$ be the cardinal sum of $\kappa$ many chains of length $10^{100}$ each, and let 
$\XX$ be the cardinal sum of $\kappa$ many  Y-posets. Then, as complete lattices, 
$\Pquleq $ is  $347$-generated,  $\Quleq(\PP\cplus \bbA_2)$ is  $345$-generated, $\Quleq(\XX)$ is $\bigl(4\cdot (2\cdot2\cdot2+1\cdot1\cdot 1)+  3+3\bigr)+3=45$-generated, and $\Quleq(\XX\cplus \bbA_2)$ is $\bigl(4\cdot (2\cdot1\cdot1+1\cdot1\cdot 1)+ 0+0\bigr)+3=15$-generated. 
\end{example}

\begin{remark}\label{rem:smRwkvvvthR} 
With few exceptions, we do not know what  the smallest possible size of a generating set of $\Pquleq$ for a  poset $\PP$ in the scope of Theorem \ref{thmmain} is.
We only know that the theorem is sharp for every antichain of size $\ncs\neq4$, see Lemma \ref{lemma:czkLn}, and for every finite chain. 
\end{remark}

\begin{remark}
The proof of our theorem heavily uses Lemma \ref{lemma:czkLn}. For infinite sets, this lemma is taken from Cz\'edli and Kulin \cite{czgkulin}. We could only deal with accessible cardinalities in \cite{czgkulin} and in the  earlier paper Cz\'edli \cite{CzGEq4}, on which  \cite{czgkulin} relies. 
This explains that  Theorem \ref{thmmain} is also restricted to accessible cardinalities.   
\end{remark}

\section{More about the dedication}\label{sect:dedic}
The reader may wonder what happens
if we drop the assumption ``$|P|$ is an accessible cardinal''  from Theorem \ref{thmmain}.  Because of this natural question,
this is the right place to make some comments on the dedication the paper begins with.

\begin{remark}\label{rem:totik}  There are several sources, including
{\tiny{\url{https://en.wikipedia.org/wiki/Vilmos_Totik}}} 
and 
{\tiny{\url{https://mta.hu/koztestuleti_tagok?PersonId=19540}}}
where one can read about Professor Vilmos Totik. As a witness, here
I add three stories; the third story sheds some light on the  question mentioned above.
   First, Vilmos  was one of my roommates in Lor\'and E\"otv\"os Students' Hostel
for (almost) three academic years from 1974 to 1977; his  excellence
was clear for all of us, fellow students, even that time.
  Second,   he received the Researcher of the Year award from
the University of Szeged in the field of physical sciences on November 9,
2019; this was the first occasion when this new award was delivered; see
\href{https://u-szeged.hu/news-and-events/2019/honorary-titles-and}{\tiny{https://u-szeged.hu/news-and-events/2019/honorary-titles-and}} .  
As this link shows, that day was splendid for us also for another reason:
Doctor Honoris Causa title was awarded to Professor Ralph McKenzie, who hardly needs any introduction to the targeted readership.
I still remember the dinner we organized for the evening of that splendid day, where
the two celebrated professors of quite different fields of mathematics
discussed, among other things,  their fishing experiences.
   Third, when (around 1999) I mentioned the basic idea of my proofs
in \cite{CzGEq4} and \cite{czgeek} to Professor Totik,  it took him less than a minute to begin (and few minutes to complete) a proof showing that my method cannot work for all infinite cardinals.
Later, when I forgot his proof, I could not find it again and I needed
a professor of set theory to find the proof again.  
(In spite of many changes in subsequent proofs, see Cz\'edli \cite{czgdaugthent}, Cz\'edli and Kulin \cite{czgkulin}, and Cz\'edli and Oluoch \cite{czgoluoch}, the basic idea of dealing with sets of large cardinalities is, unfortunately, still the same and we seem to be far from strongly inaccessible cardinals.)

To conclude the stories, I wish Professor Vilmos Totik a happy birthday, further successes, and all the best.
\end{remark}

\end{document}